\documentclass[12pt]{amsart}

\usepackage{a4wide}

\usepackage{enumerate}

\newtheorem{theorem}{Theorem}[section]

\newtheorem{proposition}[theorem]{Proposition}

\newtheorem*{maintheorem}{Main Theorem}

\theoremstyle{definition}

\newcommand{\C}{\ensuremath{\mathbb{C}}}
\newcommand{\R}{\ensuremath{\mathbb{R}}}

\newcommand{\cal}[1]{\ensuremath{\mathcal{#1}}}

\begin{document}

\title[Ruled hypersurfaces with constant mean curvature]{Ruled hypersurfaces with constant mean curvature\\in complex space forms}

%    author one information
\author[M.~Dom\'{\i}nguez-V\'{a}zquez]{Miguel Dom\'{\i}nguez-V\'{a}zquez}
\address{Instituto de Ciencias Matem\'aticas (CSIC-UAM-UC3M-UCM), Madrid, Spain.}
\email{miguel.dominguez@icmat.es}
%    author two information
\author[O.~P\'erez-Barral]{Olga P\'erez-Barral}
\address{Department of Mathematics, University of Santiago de Compostela, Spain.}
\email{olgaperez.barral@usc.es}

\thanks{Both authors have been supported by the project MTM2016-75897-P (AEI/FEDER, Spain). The first author acknowledges support by the project ED431F 2017/03 (Xunta de Galicia, Spain) and ICMAT Severo Ochoa project SEV-2015-0554 (MINECO, Spain), as well as by the European Union's Horizon 2020 research and innovation programme under the Marie Sklodowska-Curie grant agreement No.~745722. }


%\date{}

\begin{abstract}
We show that ruled real hypersurfaces with constant mean curvature in the complex projective and hyperbolic spaces must be minimal. This provides their classification, by virtue of a result of Lohnherr and Reckziegel.
\end{abstract}


\maketitle

%    Text of article.
\section{Introduction}
A classical theorem of Catalan states that the only ruled minimal surfaces in the Euclidean $3$-space are planes and helicoids. This result has been extended in several directions in the context of space forms. In particular, Barbosa and Delgado~\cite{BD:ajm} proved that there are no ruled hypersurfaces with nonzero constant mean curvature in nonflat space forms other than the $3$-sphere. This contrasts with the existence of many minimal examples, which were completely described in~\cite{BDJ}.

In the context of submanifold geometry of complex space forms, a ruled real hypersurface is a submanifold of real codimension one that is foliated by totally geodesic complex hypersurfaces. The aim of this note is to prove the complex analog of the result in~\cite{BD:ajm} mentioned above. %Thus,  the still open question of the existence of nonminimal examples with constant mean curvature in nonflat complex space forms has a negative answer.

\begin{maintheorem}
	Let $M$ be a ruled real hypersurface with constant mean curvature in a nonflat complex space form. Then, $M$ is minimal.
\end{maintheorem}

Minimal ruled hypersurfaces in the nonflat complex space forms, that is, in the complex projective and hyperbolic spaces, $\C P^n$ and $\C H^n$, have been classified by Lohnherr and Reckziegel~\cite{LR} (see also~\cite{ABM}) into three classes:
\begin{enumerate}[{\rm (i)}]
	\item Clifford cones in $\C P^n$ or $\C H^n$,
	\item bisectors in $\C H^n$, and
	\item Lohnherr hypersurfaces in $\C H^n$.
\end{enumerate}
Here, by a Clifford cone we mean a singular hypersurface constructed as a cone over a homogeneous minimal hypersurface $S^1\times S^{2n-3}$ in a geodesic sphere; see~\cite{Ki:mathann}. Equivalently, it is constructed by attaching totally geodesic complex hyperplanes $\C P^{n-1}$ (resp.\ $\C H^{n-1}$) to a geodesic in a totally real and totally geodesic $2$-plane $\R P^2$ (resp.\ $\R H^2$). A bisector in $\C H^n$ is the geometric locus of points equidistant to two given points~\cite[p.~446]{CR}. A Lohnherr hypersurface (also called a fan) is the only complete ruled hypersurface of $\C H^n$ with constant principal curvatures~\cite{LR}. It is also the unique minimal hypersurface which is homogeneous, that is, an orbit of an isometric action on $\C H^n$~\cite{BD09}. Different alternative descriptions of these examples can be found in~\cite{ABM}, \cite{GG} and~\cite{Maeda}.

Thus, our Main Theorem, combined with the results in~\cite{ABM} and~\cite{LR}, yields the classification of ruled hypersurfaces with constant mean curvature in the complex projective and hyperbolic spaces. Notice that, in the flat setting, the classification of ruled hypersurfaces with constant mean curvature follows from the results in~\cite{BDJ} and~\cite{BD:ajm}.

\section{Preliminaries}\label{sec:preliminaries}

In this section we settle some notation and terminology concerning the theory of real hypersurfaces in complex space forms. For more information, we refer to \cite{CR} and \cite{NR}.

Let $M$ be a smooth hypersurface of a Riemannian manifold $\bar{M}$. Since the arguments that follow are local, we can assume that $M$ is embedded and take a unit normal vector field $\xi$ on $M$.  We denote by $\langle\,\cdot\,,\,\cdot\,\rangle$ the metric of $\bar{M}$, by $\bar{\nabla}$ its Levi-Civita connection, and by $\bar{R}$ its curvature tensor, which we adopt with the following sign convention: $\bar{R}(X,Y)Z=[\bar{\nabla}_X,\bar{\nabla}_Y]Z-\bar{\nabla}_{[X,Y]}Z$. 
The shape operator of $M$ (with respect to the unit normal $\xi$) is the endomorphism $S$ of $TM$ defined by
\[
SX=-(\bar{\nabla}_X \xi)^\top,
\]
where $X$ is a tangent vector of $M$, and $(\cdot)^\top$ denotes the orthogonal projection of a vector onto the tangent space to $M$. The Levi-Civita connection of $M$ is denoted by $\nabla$, and is determined by the Gauss formula
\[
\bar{\nabla}_X Y=\nabla_X Y+\langle S X,Y\rangle \xi.
\]
In this paper, we will need one of the second order equations of submanifold theory, namely, the Codazzi equation, which is written as
\[
\langle\bar{R}(X,Y)Z,\xi\rangle= \langle(\nabla_X S)Y,Z\rangle-\langle(\nabla_Y S)X,Z\rangle, \qquad X,Y,Z\in \Gamma(TM).
\]
Hereafter, by $\Gamma(\cdot)$ we denote the module of smooth sections of a distribution on~$M$. 

The shape operator $S$ is a self-adjoint endomorphism with respect to the induced metric on $M$, and thus it can be diagonalized with real eigenvalues. These eigenvalues are called the principal curvatures of $M$, the corresponding eigenspaces are the principal curvature spaces, and the corresponding eigenvectors are the principal curvature vectors. If $\lambda$ is a principal curvature of $M$, we denote by $T_\lambda$ its corresponding principal curvature space. The mean curvature function of $M$ is the trace of the shape operator $S$.

From now on, we will assume that $\bar{M}= \bar{M}^n(c)$, where $\bar{M}^n(c)$ represents a complex space form of complex dimension $n$ and constant holomorphic sectional curvature $c\in \R$, that is, a complex projective space $\C P^n$ if $c>0$, a complex Euclidean space $\C^n$ if $c=0$, or a complex hyperbolic space $\C H^n$ if $c<0$. Let $J$ denote the complex structure of $\bar{M}^n(c)$. Since $\bar{M}^n(c)$ is K\"ahler, we have that $\bar{\nabla}J=0$. We will also make use of the formula of the curvature tensor $\bar{R}$ of a complex space form of constant holomorphic sectional curvature $c$:
\begin{align*}
	\langle\bar{R}(X,Y)Z,W\rangle={}&
	\frac{c}{4}\Bigl(
	\langle Y,Z\rangle\langle X,W\rangle
	-\langle X,Z\rangle\langle Y,W\rangle\\[-1ex]
	&\phantom{\frac{c}{4}\Bigl(}
	+\langle JY,Z\rangle\langle JX,W\rangle
	-\langle JX,Z\rangle\langle JY,W\rangle
	-2\langle JX,Y\rangle\langle JZ,W\rangle
	\Bigr).
\end{align*}

Let $M$ be a real hypersurface of $\bar{M}^n(c)$, that is, a submanifold with real codimension one. The tangent vector field $J\xi$ is called the Hopf or Reeb vector field of $M$. 
We define the integer-valued function $h$ on $M$ as the number of the principal curvature spaces onto which $J\xi$ has nontrivial projection or, equivalently, as the dimension of the minimal subspace of the tangent space to $M$ that contains $J\xi$ and is invariant under the shape operator $S$. Thus, $M$ is said to be Hopf if $h=1$ on $M$, that is, if $J\xi$ is a principal curvature vector field everywhere. 

A real hypersurface $M$ of $\bar{M}^n(c)$ is said to be ruled if it is foliated by totally geodesic complex hypersurfaces of $\bar{M}^n(c)$. In other words, $M$ is ruled if the orthogonal distribution to the Hopf vector field $J\xi$ is integrable and its leaves are totally geodesic submanifolds of~$\bar{M}^n(c)$. Locally, ruled hypersurfaces are embedded, but globally they may have self-intersections and singularities. See~\cite{Maeda} for more information on ruled real hypersurfaces.

\section{Proof of the Main Theorem}\label{sec:LeviCivita}

Let $M$ be a ruled real hypersurface in a nonflat complex space form $\bar{M}^n(c)$, $c\neq 0$, with (locally defined) unit normal vector field $\xi$.

\begin{proposition}\label{prop:ruled_algebraic}
Let $p\in M$ such that $h(p)\neq 1$. Then $h(p)=2$, $M$ has exactly two nonzero principal curvatures $\alpha$, $\beta$ at $p$, both of multiplicity one, and $0$ is always a principal curvature of $M$. Moreover, $J\xi_p=au+bv$ for $u\in T_\alpha(p)$, $v\in T_\beta(p)$ and $a$, $b\in \R$ such that $a^2+b^2=1$ and
\[
a^2=\frac{\alpha}{\alpha-\beta}, \qquad b^2=\frac{\beta}{\beta-\alpha}.
\]
\end{proposition}
\begin{proof}
Fix $p\in M$. By \cite[Proposition~8.27]{CR} (see also~\cite{ABM}) we know that $S(J\xi)^\perp\subset\R J\xi$, and hence, the shape operator $S$ of $M$ at $p$ satisfies
\[
SJ\xi_p=\lambda J\xi_p+\mu z, \qquad Sz=\mu J\xi_p , \qquad S w=0,
\]
for a certain unit vector $z\in T_p M$ orthogonal to $J\xi_p$, and for all $w\in T_pM$ perpendicular to $J\xi_p$ and $z$. Let $\alpha$ and $\beta$ be the eigenvalues of $S$ restricted to the invariant subspace $\R J\xi_p\oplus\R z$, and $u$, $v$ some corresponding orthogonal eigenvectors of unit length. Thus, one can write $J\xi_p=a u+b v$ for some $a$, $b\in\R$ with $a^2+b^2=1$. Moreover, $a\neq 0\neq b$ since $h(p)\neq 1$ by assumption and, hence, we must have $h(p)=2$. Then 
\[
\lambda=\langle SJ\xi_p,J\xi_p\rangle=\langle a\alpha u+b\beta v,au+bv\rangle=a^2\alpha+b^2\beta.
\]
Moreover, we have that $\lambda=\alpha+\beta$ due to the invariance of the trace of $S$. Both equations imply that $\alpha\neq \beta$, since otherwise this would give us that $\alpha=\lambda=2\alpha$, producing $\alpha=\beta=0$, which contradicts $h(p)\neq 1$. Finally, combining again both equations with $a^2+b^2=1$ we obtain the formulas for $a^2$ and $b^2$ in the statement.
\end{proof}

It is known that ruled hypersurfaces in a nonflat complex space form cannot be Hopf \cite[Remark~8.44]{CR}. This implies that no open subset of a ruled hypersurface $M$ in $\bar{M}^n(c)$, $c\neq 0$, is Hopf. Thus, by virtue of Proposition~\ref{prop:ruled_algebraic}, $h=2$ on an open and dense subset $\cal{U}$ of $M$. 
Again by Proposition~\ref{prop:ruled_algebraic} we know that, at each point, $\cal{U}$ has exactly two distinct nonzero principal curvatures. Altogether, $\cal{U}$ has exactly two nonzero principal curvature functions $\alpha$ and $\beta$, both of multiplicity one at every point. We also have $J\xi=aU+bV$ for some unit vector fields $U\in\Gamma(T_\alpha)$ and $V\in \Gamma(T_\beta)$ and smooth functions $a$, $b\colon \cal{U}\to \R$ with $a^2+b^2=1$. Let $k$ denote the constant mean curvature of $M$, so $k=\alpha+\beta$ on $\cal{U}$. Thus, again by Proposition~\ref{prop:ruled_algebraic}, we have that
\begin{equation}\label{eq:a^2,b^2}
a^2=\frac{\alpha}{2\alpha-k}, \qquad b^2=\frac{k-\alpha}{k-2\alpha}.
\end{equation}
In particular, since $h=2$ on $\cal{U}$, Equation \eqref{eq:a^2,b^2} implies $\alpha\neq 0\neq k-\alpha$ at every point of~$\cal{U}$.

From now on we will work on the open and dense subset $\cal{U}$ of $M$.

\begin{proposition}\label{prop:algebra}
	With the notation above, there exists $A\in\Gamma(T_0)$ such that
	\begin{align*}
		J\xi &= aU+bV, & JU  &= -bA-a\xi,\\
		JA &= bU-aV, & JV  &= aA-b\xi.
	\end{align*}
\end{proposition}
\begin{proof}
The proof of~\cite[Lemma~3.1]{DD:indiana} (see also~\cite[Proposition~3.1]{DDV:annali}) adapts to our setting with very minor modifications.
\end{proof}

\begin{proposition}\label{prop:Levi-Civita}
	The Levi-Civita connection of $\cal{U}$ satisfies the following equations:
	\begin{align*}
		\nabla_U U  &=-\frac{a b (c+8 \alpha  (k-\alpha ))}{4 \alpha }A,
		&
		\nabla_U V  &=\frac{c+4 k \alpha }{4 (2 \alpha -k)}A,
		\\
		\nabla_V V  &=\frac{a b (c+8 \alpha  (k-\alpha ))}{4 (k-\alpha )}A,
		&
		\nabla_V U  &=\frac{c+4 k (k-\alpha )}{4 (2 \alpha-k )}A.
	\end{align*}
	Moreover:
\begin{equation}\label{eq:deriv_functions}
U\alpha=V\alpha=0, \qquad \text{and}\qquad A\alpha=\frac{ab}{2}(c+4 \alpha  (\alpha -k)).
\end{equation}

\end{proposition}
\begin{proof}
	Using the fact that $U$ and $V$ are orthogonal eigenvectors of the shape operator $S$ with respective eigenvalues $\alpha$ and $\beta=k-\alpha$, we have
\begin{align*}
\langle (\nabla_U S)V, U\rangle&=\langle \nabla_U SV-S\nabla_U V, U\rangle=\langle \nabla_U((k-\alpha) V),U\rangle-\alpha\langle\nabla_U V, U\rangle=
\\
&=  \langle -(U\alpha)V+(k-\alpha)\nabla_UV,U\rangle-\alpha\langle\nabla_UV,U\rangle= (k-2\alpha)\langle\nabla_UV,U\rangle.
\end{align*}
Similarly, using the fact that $\langle \nabla_V U,U\rangle=0$ since $U$ has constant length, we have that
\[
\langle (\nabla_V S)U,U\rangle= \langle \nabla_V SU-S\nabla_V U,U\rangle= \langle \nabla_V(\alpha U),U\rangle -\alpha\langle \nabla_V U,U\rangle=V\alpha.
\]
Moreover, the expression for the curvature tensor of a complex space form, taking into account the relations in Proposition~\ref{prop:algebra}, yields $\langle\bar{R}(U,V)U,\xi\rangle=0$. Therefore, the Codazzi equation applied to the triple $(U,V,U)$ gives us
\begin{equation}\label{eq:cod_uvu}
0=(k-2\alpha)\langle \nabla_UV,U\rangle -V\alpha.
\end{equation}
Analogously, the Codazzi equation applied to the triple $(V,U,V)$ yields
\begin{equation}\label{eq:cod_vuv}
0=(k-2\alpha)\langle \nabla_VU,V\rangle -U\alpha.
\end{equation}
Now, using the fact that $\bar{\nabla}J=0$, the definition of the shape operator, the relation $J\xi=aU+bV$, $\langle U,\nabla_UU\rangle=0$ and $0=U\langle V,U\rangle=\langle \nabla_U V,U\rangle+\langle V,\nabla_UU\rangle$, we have
\begin{align*}
	Ua&=U\langle J\xi, U\rangle=\langle \bar{\nabla}_U J\xi,U\rangle+\langle J\xi,\nabla_U U\rangle = \langle J\bar{\nabla}_U\xi,U\rangle+\langle J\xi,\nabla_U U\rangle
	\\
	&=\langle SU,JU\rangle+\langle aU+bV,\nabla_UU\rangle = b\langle V,\nabla_UU\rangle=-b\langle \nabla_U V,U\rangle.	
\end{align*}
By multiplying this relation by $2a$ and taking into account that 
\[
2aUa=U(a^2)=U\left(\frac{\alpha}{2\alpha-k}\right)=-\frac{kU\alpha}{(2\alpha-k)^2},
\]
we deduce that
\begin{equation}\label{eq:deriv_ua}
0=kU\alpha-2ab(2\alpha-k)^2\langle\nabla_UV, U\rangle.
\end{equation}
Analogously, the relation $Va=V\langle J\xi, U\rangle$ implies 
\begin{equation}\label{eq:deriv_va}
0=kV\alpha+2ab(2\alpha-k)^2\langle \nabla_VU,V\rangle.
\end{equation}
Equations~\eqref{eq:cod_uvu}, \eqref{eq:cod_vuv}, \eqref{eq:deriv_ua} and \eqref{eq:deriv_va} constitute a homogeneous linear system in the unknowns $\langle \nabla_UV,U\rangle$, $\langle \nabla_VU,V\rangle$, $U\alpha$ and $V\alpha$. Some calculations using \eqref{eq:a^2,b^2} show that the determinant of the matrix of this system is $(k-2\alpha)^4\neq 0$. Hence, the four unknowns must vanish, which implies
\begin{equation}\label{eq:nabla_uv}
	\langle \nabla_UV, U\rangle =\langle \nabla_VU,V\rangle=\langle \nabla_UU, V\rangle =\langle \nabla_VV,U\rangle=0, \qquad U\alpha=V\alpha=0.
\end{equation}
In particular, we obtain the first equations in~\eqref{eq:deriv_functions}.

Now, using similar arguments as above one shows that the Codazzi equation for the triple $(U,A,U)$ reads
\begin{align*}
0&=\langle\bar{R}(U,A)U,\xi\rangle-\langle (\nabla_U S)A, U\rangle+\langle (\nabla_A S)U, U\rangle= -\frac{3}{4}abc+\alpha\langle \nabla_U A, U\rangle+A\alpha,
\end{align*}
from where we get
\begin{equation}\label{eq:cod_uau}
\langle \nabla_U A, U\rangle=\frac{1}{4\alpha}(3abc-4A\alpha).
\end{equation}
Similarly, from the Codazzi equation for the triple $(V,A,V)$ we obtain
\begin{equation}\label{eq:cod_vav}
\langle \nabla_V A, V\rangle=\frac{1}{4(k-\alpha)}(-3abc+4A\alpha).
\end{equation}
On the other hand, using $\bar{\nabla}J=0$, Proposition~\ref{prop:algebra} and \eqref{eq:cod_uau} we have
\begin{align}\label{eq:deriv_u0}
0&=U\langle J\xi, A\rangle=\langle \bar{\nabla}_U J\xi,A\rangle+\langle J\xi, \nabla_U A\rangle= \langle SU, JA\rangle+\langle aU+bV, \nabla_U A\rangle
\\ \nonumber
&=\alpha b +a\langle U,\nabla_U A\rangle+ b\langle V, \nabla_U A\rangle=\alpha b +\frac{a}{4\alpha}(3abc-4A\alpha)- b\langle \nabla_U V, A\rangle,
\end{align}
and similarly, using~\eqref{eq:cod_vav},
\begin{equation}\label{eq:deriv_v0}
0=V\langle J\xi, A\rangle=(\alpha-k) a +\frac{b}{4(k-\alpha)}(-3abc+4A\alpha)- a\langle  \nabla_V U,A\rangle.
\end{equation}
Writing the Codazzi equation for the triple $(U,V,A)$, and using the expressions for $\langle \nabla_U V,A \rangle$ and $\langle \nabla_V U,A\rangle$ given by \eqref{eq:deriv_u0} and \eqref{eq:deriv_v0}, together with \eqref{eq:a^2,b^2}, we obtain
\begin{align*}
0&=\langle \bar{R}(U,V)A,\xi\rangle-\langle (\nabla_U S)V,A\rangle+\langle(\nabla_V S)U,A\rangle= -\frac{c}{4}+(\alpha-k)\langle \nabla_U V, A\rangle+\alpha\langle \nabla_V U, A\rangle
\\
&=-\frac{c}{4}+(\alpha-k)\left(\alpha  +\frac{3a^2c}{4\alpha}-\frac{aA\alpha}{b\alpha}\right)+\alpha\left(\alpha-k -\frac{3b^2c}{4(k-\alpha)}+\frac{bA\alpha}{a(k-\alpha)}\right)
\\
&=\left(\frac{\alpha^2b^2+(\alpha-k)^2a^2}{ab\alpha(k-\alpha)}\right)A\alpha+\frac{c}{2}+2\alpha(\alpha-k)=-\frac{A\alpha}{ab}+\frac{c}{2}+2\alpha(\alpha-k),
\end{align*}
from where we get the last formula in~\eqref{eq:deriv_functions}. Inserting the expression for $A\alpha$ in \eqref{eq:deriv_functions} into the equations~\eqref{eq:cod_uau}, \eqref{eq:cod_vav}, \eqref{eq:deriv_u0} and \eqref{eq:deriv_v0}, and using~\eqref{eq:a^2,b^2}, we get
\begin{align}\label{eq:nabla_uva}
\langle \nabla_U U, A\rangle &= -\frac{a b (c+8 \alpha  (k-\alpha ))}{4 \alpha } ,&& \langle \nabla_U V, A\rangle = \frac{c+4 k \alpha }{4 (2 \alpha -k)},
\\\nonumber
\langle \nabla_V U, A\rangle &= \frac{c+4 k (k-\alpha )}{4 (2 \alpha -k)} ,&& \langle \nabla_V V, A\rangle =\frac{a b (c+8 \alpha  (k-\alpha ))}{4 (k-\alpha )}.
\end{align}

Now, let $W\in\Gamma(T_0\ominus\R A)$ be a unit vector field in the $0$-principal curvature distribution. Then, the Codazzi equation applied to the triples $(U,W,U)$ and $(V,W,V)$ yields
\begin{equation}\label{eq:cod_uwu,vwv}
\alpha\langle \nabla_U W,U\rangle+W\alpha=0, \qquad (k-\alpha)\langle \nabla_V W,V\rangle-W\alpha=0.
\end{equation}
Expanding the equations $0=U\langle J\xi, W\rangle$ and $0=V\langle J\xi, W\rangle$ in the same way as we have done several times above, and using the expressions for $\langle \nabla_U W,U\rangle$ and $\langle \nabla_V W,V\rangle$ given by \eqref{eq:cod_uwu,vwv} we obtain
\begin{equation}\label{eq:deriv_wv,wu}
0=b\langle V,\nabla_U W\rangle-\frac{aW\alpha}{\alpha}, \qquad 0=a\langle U,\nabla_V W\rangle+\frac{bW\alpha}{k-\alpha}.
\end{equation}
Finally, inserting the expressions for $\langle \nabla_UV, W\rangle$ and $\langle \nabla_VU, W\rangle$ given by the equations \eqref{eq:deriv_wv,wu} into the Codazzi equation applied to the triple $(U,V,W)$ and using \eqref{eq:a^2,b^2} we derive, after some calculations, that $W\alpha=0$. This, together with \eqref{eq:cod_uwu,vwv} and \eqref{eq:deriv_wv,wu}, implies
\[
\langle \nabla_U U, W\rangle = \langle \nabla_V V, W\rangle=\langle \nabla_U V, W\rangle=\langle \nabla_V U, W\rangle=0.
\]
Since $W\in\Gamma(T_{0}\ominus\R A)$ is arbitrary, these relations,  \eqref{eq:nabla_uv} and \eqref{eq:nabla_uva} yield the covariant derivatives in the statement of the proposition.
\end{proof}

Now we can conclude the proof of the Main Theorem.

According to the discussion before Proposition~\ref{prop:algebra}, there is an open and dense subset $\cal{U}$ of $M$ where $h=2$. Propositions~\ref{prop:algebra} and~\ref{prop:Levi-Civita} hold on this open subset. By \eqref{eq:deriv_functions}, the definition of the Lie bracket of $M$ yields
\begin{equation}\label{eq:[u,v]}
[U,V](\alpha)=U(V\alpha)-V(U\alpha)=0.
\end{equation}
On the other hand, using the fact that the Levi-Civita connection of $M$ is torsion-free, and inserting the expressions for $\nabla_UV$, $\nabla_VU$ and $A\alpha$ given in Proposition~\ref{prop:Levi-Civita}, we obtain
\begin{align}\label{eq:torsion_free}\nonumber
[U,V](\alpha)&=(\nabla_UV)\alpha-(\nabla_VU)\alpha=\left(\frac{c+4k\alpha}{4(2\alpha-k)}\right)A\alpha+\left(\frac{c+4k(k-\alpha)}{4(k-2\alpha)}\right)A\alpha
\\
&=kA\alpha=\frac{kab}{2}(c+4\alpha(\alpha-k)).
\end{align}
Assume that the mean curvature $k$ of $M$ is nonzero. Then, since on $\cal{U}$ we have $a\neq 0\neq b$, \eqref{eq:[u,v]} and~\eqref{eq:torsion_free} imply that $c+4\alpha(\alpha-k)=0$, from where we deduce that $\alpha$ is constant on $\cal{U}$ and, by the density of $\cal{U}$ on $M$, also on $M$. Therefore, $M$ has constant principal curvatures. Then, the work of Lohnherr and Reckziegel \cite[Remark~5]{LR} (or the classification of real hypersurfaces with constant principal curvatures and $h=2$ ~\cite{DD:indiana}, together with their explicit principal curvatures~\cite{BD09}) implies that $M$ must be congruent to the Lohnherr hypersurface in $\C H^n$, which is minimal. Hence, $k=0$, which gives us the desired contradiction.

\end{document}